\documentclass[a4paper,11pt]{amsart}

\usepackage{amsmath,amsthm,amssymb,verbatim,amscd,epsfig,esint}
\usepackage[latin1]{inputenc}
\usepackage[T1]{fontenc}
\usepackage{enumerate}
\usepackage{graphicx}

\theoremstyle{plain}
\newtheorem{thm}{Theorem}[section]
\newtheorem{lem}[thm]{Lemma}
\newtheorem{cor}[thm]{Corollary}
\newtheorem{prop}[thm]{Proposition}

\theoremstyle{definition}

\newtheorem{rem}[thm]{Remark}

\numberwithin{equation}{section}

\newcommand{\C}{\mathbb C}
\newcommand{\D}{\mathbb D}

\newcommand{\R}{\mathbb R}

\newcommand{\cB}{\mathcal B}
\newcommand{\cC}{\mathcal C}

\newcommand{\cS}{\mathcal S}
\newcommand{\cH}{\mathcal H}

\newcommand{\zbar}{\bar{z}}

\renewcommand{\Re}{\operatorname{Re}}

\newcommand{\loc}{\mathrm{loc}}

\let\swap=\phi
\let\phi=\varphi
\let\varphi=\swap
\let\swap=\epsilon
\let\epsilon=\varepsilon
\let\varepsilon=\swap

\let\swap=\leq
\let\leq=\leqslant
\let\leqslant=\swap
\let\swap=\geq
\let\geq=\geqslant
\let\geqslant=\swap

\newcommand{\hhh}{\mathcal{H}}

\newcommand{\cc}{\mathbb{C}}
\newcommand{\bz}{\bar{z}}
\newcommand{\dist}{\mathrm{dist}}

\renewcommand{\bar}[1]{\overline{#1}} 

\title [Nonlinear Beltrami operators]
{Nonlinear Beltrami operators, Schauder estimates and bounds for the Jacobian}\date{}

\author[K. Astala]{Kari Astala}
\address{Department of Mathematics and Statistics, University of Helsinki, 
         P.O. Box 68, FI-00014, Helsinki, Finland; LE STUDIUM, Loire Valley Institute for Advanced Studies, Orl\'eans \& Tours, France, MAPMO, rue de Chartres, 45100 Orl\'eans, France}
\email{kari.astala@helsinki.fi}

\author[A. Clop]{Albert Clop}
\address{Department of Mathematics, Universitat Aut\`onoma de Barcelona, 08193 Bellaterra (Barcelona), Catalonia}
\curraddr{}
\email{albertcp@mat.uab.cat}

\author[D. Faraco]{Daniel Faraco}
\address{Department of Mathematics, Universidad Aut\'onoma de Madrid, 28049 Ma\-drid, 
Spain; ICMAT CSIC-UAM-UCM-UC3M, 28049 Madrid, Spain}
\curraddr{}
\email{daniel.faraco@uam.es}

\author[J. J\"a\"askel\"ainen]{Jarmo  J\"a\"askel\"ainen}
\address{Department of Mathematics and Statistics, University of Helsinki, 
         P.O. Box 68, FI-00014, Helsinki, Finland; Department of Mathematics, Universidad Aut\'onoma de Madrid, 28049 Madrid, 
Spain}
\curraddr{}
\email{jarmo.jaaskelainen@helsinki.fi}

\author[A. Koski]{Aleksis Koski}
\address{Department of Mathematics and Statistics, University of Helsinki, 
         P.O. Box 68, FI-00014, Helsinki, Finland}
\curraddr{}
\email{aleksis.koski@helsinki.fi}

\thanks{K.A. was supported by Academy of Finland project  SA-12719831. 
A.C. was supported by research grants 2014SGR75 (Generalitat de Catalunya), MTM2016-75390-P (Gobierno de Espa\~na) and FP7-607647 (European Union).  D.F. was supported by research grant MTM2011-28198 from the Ministerio de Ciencia e Innovaci\'on (MCINN), by MINECO: ICMAT Severo Ochoa project SEV-2011-0087, and by the ERC 301179.
J.J. was supported by the ERC 301179 and Academy of Finland (no. 276233).
A.K. was supported by the V\"ais\"al\"a Foundation.}


\begin{document}

\frenchspacing

\begin{abstract}
We provide Schauder estimates for nonlinear Beltrami equations and lower bounds of the Jacobians for homeomorphic solutions. 
The results were announced in 
\cite{ACFJ} but here we give detailed proofs.  
\end{abstract}

\maketitle

\section{Introduction}

\noindent  This note is devoted to establish  properties of solutions to the nonlinear Beltrami equation 
\begin{equation}\label{hqrjac}
\partial_{\zbar} f(z) = \cH(z, \partial_z f(z))  \qquad \text{a.e.}
\end{equation}
under additional regularity of $\cH$. Recall that  the strong ellipticity of the equation is encoded in the fact that the function $\cH(z,\xi)$ is $k$-Lipschitz on its second variable
where $k<1$. In particular, $W^{1,2}_{\loc}$-solutions to \eqref{hqrjac} are {\it a priori} $K$-quasiregular, where $K = \frac{1+k}{1-k}$.

In the recent monograph \cite{AIM} on quasiconformal mappings and elliptic equations it was established that the nonlinear Beltrami equation  governs effectively all nonlinear planar elliptic systems. The nonlinear equation was introduced by Bojarski and Iwaniec in 
\cite{Boj74, boj, iw} and its basic $L^p$-properties were obtained in \cite{AIS}. On the other hand,   to study oscillating properties of sequences of gradients of Sobolev mappings in \cite{Faraco04,tartar}, it was vital to associate to them a corresponding nonlinear Beltrami equation.  

The nonlinear Beltrami equation shares the existence properties of homeomorphic solutions with the linear one \cite{AIM} but, for example, the uniqueness fails in general as proved in \cite{ACFJS}.  In \cite{ACFJ} it was proved that  the set of homeomorphic solutions forms an embedded submanifold of $W^{1,2}_{\textrm{loc}}(\C, \C)$ and that under regularity assumptions the manifold of homeomorphic solutions defines
uniquely the structure function $\cH$.  The arguments in \cite{ACFJ} rely on regularity properties of the solutions, which we prove in the current paper. 

Let us state our regularity assumptions on the structure function $\cH(z,\xi)$. Throughout this paper we will assume H\"older continuity of $\cH$ in the first variable and $k$-Lipschitz dependence on the second one. More precisely, given an open bounded set $\Omega \subset \C$, we assume that 
\begin{equation}\label{Hcondition}
\aligned
&|\cH(z_1, \xi_1) - \cH(z_2, \xi_2)| \leq \mathbf{H}_{\alpha}(\Omega)\,|z_1 - z_2|^{\alpha}\bigl(|\xi_1| + |\xi_2|\bigr) + k\,|\xi_1 - \xi_2|,\\
&\cH(z_1, 0) \equiv 0,
\endaligned
\end{equation}
for all $z_1, z_2 \in \Omega$, $\xi_1, \xi_2 \in \C$, where $\alpha \in (0,1)$ and  $k = \frac{K - 1}{K + 1} < 1$ are fixed. 

In  case $\cH(z,\xi)$ is linear in the second variable, \eqref{Hcondition} implies 
that the derivatives of the solutions to Beltrami equation  are $\alpha$-H\"older continuous and that the Jacobian of a homeomorphic solution
does not vanish (see \cite{AIM, ren}). Our goal is to see if similar regularity results hold in the general nonlinear case. 
We start with our second main question.

\begin{thm}\label{Jac} Suppose the structure function $\hhh(z,\xi)$ satisfies \eqref{Hcondition}.  Then a homeomorphic solution  $f \in W^{1,2}_{\loc}(\Omega, \C)$ to the nonlinear Beltrami equation
\eqref{hqrjac} has a positive Jacobian, $J(z, f) > 0$.

Further, if $\Omega = \C$ and $f:\cc \to \cc$ is a normalised solution, i.e., $f(0) = 0$ and $f(1) = 1$, then there is a lower bound for the Jacobian
$$ \inf_{z\in\D(0, R_0)} J(z, f) \geq c(\cH, R_0) > 0, \qquad 0 < R_0 < \infty. $$
\end{thm}

Besides of intrinsic interest, the non-vanishing of the Jacobian is, e.g., a key property needed in the study of manifolds of quasiconformal maps in  \cite{ACFJ}. 
In the linear case the statement can be shown by the representation theorem of the quasiregular maps (e.g., \cite[Theorems II.5.2 and II.5.47]{ren}) or by using the Schauder estimates for the inverse (e.g., proof of \cite[Proposition~5.1]{G-cl}), i.e., showing that also  $f^{-1}$ solves a Beltrami equation with H\"older continuous coefficients and hence the inverse is locally H\"older continuous, too. In the nonlinear case it is much harder to establish a suitable equation for the inverse. If we denote $g = f^{-1}$ then $g$ satisfies the nonlinear Beltrami equation
$$
\partial_{\bar{\omega}}g(\omega) = -\frac{1}{J(z, f)}\cH\left(g(\omega), J(z, f)\,\overline{\partial_{\omega}g(\omega)}\right), \qquad \omega = f(z) \qquad \text{a.e.},
$$
which would have   H\"older continuous coefficients if we {\em a priori}  knew that the Jacobian $J(z, f)$ has a positive lower bound. 

 In Section~\ref{Jacproof} we show that it is also possible to recover a nonlinear equation (that satisfies \eqref{Hcondition}) for $g$, giving us the required regularity to be able to conclude that the Jacobian must be positive everywhere.

\medskip

Next we turn to the regularity of the gradient. 
Nowadays the term Schauder estimates refers to various types of  H\"older regularity results in the theory of PDEs. Juliusz Schauder pioneered these topics in \cite{s1, s2}. His papers deal mostly with linear, quasilinear and nonlinear elliptic equations of second order. The importance of his ideas (freezing the equation, i.e., viewing equations with H\"older regular coefficients locally as perturbations of equations with constant coefficients) is reflected in an enormous  number of  applications and generalisations. These ideas were successfully used to deal with   quasilinear equations in \cite{LU} and the nonlinear divergence equations with $C^1$-dependence on the gradient variable \cite[Chapter 6]{Gia}, see also \cite{MingioneKuusi} for recent developments.  Notice that quasilinear elliptic equations relate to the nonlinear Beltrami equation through the two dimensional Hodge operator \cite{AIM}, though the relation to the regularity of $\cH$ is not clear.

Schauder estimates for general nonlinear structure functions $\cH(z, \xi)$, which are only Lipschitz in the gradient variable $\xi$ and H\"older continuous in $z$ form an important  step in proving Theorem \ref{Jac}. The required estimates do not seem to appear in literature in this generality, and therefore  we give a  quasiregular proof for the  Schauder estimates in this setting.  A different quasiregular approach of Schauder estimates for linear and quasilinear Beltrami equations  is considered in \cite[Chapter 15]{AIM}.

\begin{thm}\label{schauder} Assuming \eqref{Hcondition}, 
suppose   $f \in W^{1,2}_{\loc}(\Omega, \C)$ is  a solution to the nonlinear Beltrami equation
\begin{equation*}
\partial_{\zbar} f(z) = \cH(z, \partial_z f(z))  \qquad \text{a.e. \, in} \;\;  \Omega.
\end{equation*}
Then  $f \in C^{1, \gamma}_{\loc}(\Omega, \C)$, where $\gamma = \alpha$, if $\alpha < \frac1K$; otherwise 
one  can take any $\gamma < \frac1K$. 
Moreover, we have a norm bound, when  $\D(\omega, 2R) \Subset \Omega$,
\begin{equation}\label{thmnorm}
{\|D_zf\|}_{C^\gamma(\D(\omega, R))}
\leq c(K, \alpha, \gamma, \omega, R, \mathbf{H}_\alpha(\Omega))\,\|D_z f\|_{L^2(\D(\omega, 2R))}.\end{equation}
\end{thm}

Let us emphasise that there is a restriction $\gamma < \frac 1K$ on the H\"older exponent. This restriction already occurs at the level of the autonomous equation 
\[ \partial_{\zbar} f(z) = \cH( \partial_z f(z)) \qquad \text{a.e.,} \]
for which we prove the $C^{1,\frac1K}_{\loc}$-regularity of solutions in Corollary~\ref{holder1K}. We do not know whether the bound $\gamma < \frac1K$ of Theorem \ref{schauder} is sharp, but the example
\[f_0:\cc \to \cc, \quad f_0(z) = z^2|z|^{\frac{3}{2K + 1} - 1}\]
shows that we must at least have the bound $\gamma \leq \frac{3}{2K + 1}$, as the function $f_0$ solves an autonomous equation with ellipticity constant $k=\frac{K-1}{K+1}$. This surprisingly shows that the optimal H\"older exponent in Theorem~\ref{schauder} depends not only on $\alpha$, but on the ellipticity constant of $\hhh$ as well. In particular one cannot always take $\gamma = \alpha$, contrary to the case where $\cH(z,\xi)$ is linear in $\xi$. We leave it as an interesting open problem to determine the optimal H\"older exponent in terms of $K$.

The restriction $\gamma < \frac1K$ is not needed if in addition to \eqref{Hcondition}, the structure function $\cH$ is assumed to be $C^1$ in the gradient variable as well. This will follow from the fact that for a $C^1$-regular autonomous equation, the solutions will be shown to be in $C^{1,\beta}_{\loc}(\Omega, \C)$ for every $0 < \beta < 1$. The  estimate on the $C^{1,\beta}_{\loc}(\Omega, \C)$-norm is locally uniform in the $L^2$-norm, but the dependence is not linear (as it is in \eqref{thmnorm}). 

\begin{thm}\label{schauderC1}
Let  $f \in W^{1,2}_{\loc}(\Omega, \C)$ be a solution to the  nonlinear Beltrami equation \eqref{hqrjac}, where we assume in addition that $\xi \mapsto \cH(z, \xi) \in C^1(\C, \C)$. Then  $f \in C^{1, \alpha}_{\loc}(\Omega, \C)$ with $\alpha$ as in \eqref{Hcondition}.
\end{thm}

We will first study the autonomous case (Section~\ref{autosec}) and then  in the spirit of Schauder's estimates tackle the general case by perturbation.
The proof of Theorem~\ref{schauder} will be given in Section~\ref{schaudersec} and Theorem~\ref{schauderC1} is considered in Section~\ref{C1sec}.

\subsection*{Acknowledgements} The authors thank Professor Tadeusz Iwaniec for  fruitful discussions and for sharing his ideas on finding the nonlinear equation for the inverse map.

\section{Schauder-type estimates}

\subsection{Autonomous equation and integral estimates}\label{autosec}

We start with an auxiliary result for the nonlinear Beltrami equation with constant coefficients (see \cite{Sverak93, Faraco04, FaracoKristensen12}). In this case  $\cH$ depends only on the gradient variable, and the requirement \eqref{Hcondition} reduces to $\cH(0)=0$ with $|\cH(\xi_1)-\cH(\xi_2)| \leq k|\xi_1-\xi_2|$.

\begin{prop}\label{ccc}
Let $F  \in W^{1,2}_{\loc}(\Omega, \C)$ be a solution to the  autonomous  nonlinear Beltrami equation 
\begin{equation}\label{auto}
\partial_{\zbar} F(z) = \cH(\partial_z F(z))  \qquad \text{for a.e. $z\in \Omega$}.
\end{equation}
Then the directional derivatives of $F$ are $K$-quasiregular, $K = \frac{1 + k}{1 - k}$.
\end{prop}

\begin{proof}
Let $h > 0$. The difference quotients 
$$
F_h(z) := \frac{F(z + he) - F(z)}{h}, \qquad |e| = 1
$$
are $K$-quasiregular. Indeed, by \eqref{auto},
\begin{equation}\label{qrforDF}
\begin{aligned}
|\partial_{\zbar} F_h(z)| &= \left| \frac{\cH(\partial_z F(z + he)) - \cH( \partial_z F(z))}{h}\right|\\ &\leq k\,\frac{|\partial_z F(z + he) - \partial_z F(z)|}{|h|}
= k\,|\partial_z F_h(z)|.
\end{aligned}
\end{equation} 
Now, we have a Caccioppoli estimate for $F_h$, see e.g. \cite[Theorem 5.4.2]{AIM}. For $\rho < R$ and any constant $c$
\begin{equation}\label{caccioppoli1}
\int_{\D_{\rho}} |D_z F_h|^2 \leq  \frac{c(K)}{(R - \rho)^2}\int_{\D_{R}} |F_h - c|^2,
\end{equation}
where we denote $\D_r = \D(z_0, r)$. 
Thus $c(K)\fint_{\D_{R}} (|D_z F|^2 + 1)$ is a uniform bound for the derivative of the difference quotient for the range $0 < \rho \leq \frac{R}{2}$. Hence the directional derivative $\partial_e F \in W^{1, 2}_{\loc}(\Omega, \C)$. Further, letting $h \to 0$ in \eqref{qrforDF}, we see that $\partial_e  F(z)$ is $K$-quasiregular.
\end{proof}
Therefore,  the directional derivatives inherit the properties of $K$-quasiregular maps. We will need few integral estimates that we prove next.

\begin{prop}\label{integralestimate}
Let $g \in W^{1, 2}_{\loc}(\Omega, \C)$ be $K$-quasiregular. Then 
\begin{equation}\label{g1}
\| g \|_{L^2(\D(z_0, \rho))} \leq c(K)\,\frac{\rho}{R}\;\| g \|_{L^2(\D(z_0, R))}
\end{equation}
for $\D(z_0, \rho) \subset \D(z_0, R) \subset \Omega$.
Moreover, $g$ is locally $\frac1K$-H\"older continuous; formulated in a Morrey-Campanato form we have  
\begin{equation}\label{normbound}
\|g-g_\rho\|_{L^2(\D(z_0,\rho))} \leq  c(K)\left(\frac\rho{R}\right)^{1 + \frac1K}\,\|g-g_{R}\|_{L^2(\D(z_0,R))}
\end{equation}
for any $\rho\leq R$, where $g_r = \fint_{\D(z_0, r)} g$. 
\end{prop}

\begin{proof} We start by proving \eqref{g1}. Denote $\D_r = \D(z_0, r)$. Since $g$ is $K$-quasiregular,  we have by  Caccioppoli's inequality and weak reverse H\"older inequalities, \cite[Theorem 5.4.2]{AIM}, \cite[Proposition 1]{nolder}, for $\frac{2K}{K+1} < p < \frac{2K}{K-1}$,
\begin{equation}\label{cacci}
\|D_z g \|_{L^p(\D_{R/2})} \leq c_0(p, K, R)\, \| g \|_{L^p(\D_{2R/3})} \leq c_1(p, K, R)\,\| g \|_{L^2(\D_R)}.
\end{equation}
Now, for $\rho \leq \frac{R}{2}$,
\begin{align*}
\|g\|_{L^2(\D_\rho)} &\leq \sqrt{\pi}\,\rho\,\sup_{\D_\rho}|g| \leq c(R)\,\rho\,\|g \|_{W^{1,p}(\D_{R/2})}\\
& \leq c(p, K, R)\,\rho\, \|g\|_{L^2(\D_R)},
\end{align*}
where the second to the last inequality follows from the Sobolev embedding, by Morrey's inequality (choose $p > 2$), and the last one from Caccioppoli's inequality \eqref{cacci}. By rescaling, one sees that $c(p, K, R) = c(p, K)R^{-1}$. Hence
\begin{equation*}
\| g \|_{L^2(\D(z_0, \rho))} \leq c(p, K)\,\frac{\rho}{R}\;\| g \|_{L^2(\D(z_0, R))},
\end{equation*}
for $\rho \leq R$ (above we show the estimate for $\rho \leq  \frac{R}{2}$ and it is trivial for  $\frac{R}{2} < \rho \leq R$ with possibly a bigger constant). We have thus proved the integral estimate \eqref{g1}.

\medskip

The $\frac{1}{K}$-H\"older inequality of $K$-quasiregular maps goes back to Morrey \cite[Section 3.10]{mor, AIM}. For later purposes we recall how this follows  using the isoperimetric inequality for Sobolev spaces, in 
combination with Caccioppoli's and Poincar\'e's inequalities, and the pointwise equivalence of $|D_z g(z)|^2$ and $J(z, g)$ for quasiregular maps.  

We have, by the isoperimetric inequality in the Sobolev space and the H\"older inequality, that the mapping $\psi(r) :=  r^{-\frac2K}\int_{\D_r}J(z, g)\,dA(z)$ is non-decreasing.  Indeed, \begin{equation*}\label{nondec}
\aligned
&\int_{\D_r}J(z, g)\,dA(z) \leq \frac{1}{4\pi}\left(\int_{\partial \D_r}|D_z g(z)|\,dz\right)^2\\
&\quad\leq \frac{K|\partial \D_r|}{4\pi}\int_{\partial \D_r}\frac{|D_z g(z)|^2}{K}\,dz \leq \frac{Kr}{2}\int_{\partial \D_r}J(z, g)\,dz,
\endaligned
\end{equation*}
for $\D_r \subset \Omega$, where the last inequality follows by quasiregularity (i.e., $|D_z g(z)|^2 \leq K J(z, g)$ almost everywhere). In other words $\psi'(r) \geq 0$.

The non-decreasing of $\psi$ implies that 
\begin{equation}\label{Jacgrow}
\int_{\D_\rho}J(z, g)\,dA(z)\leq \left(\frac{\rho}{R}\right)^{2/K}\int_{\D_R}J(z, g)\,dA(z),
\end{equation}
for $\D_\rho \subset \D_R \subset \Omega$. Now, by Poincar\'e's inequality, $K$-quasiregularity, \eqref{Jacgrow}, and  Caccioppoli's inequality \eqref{caccioppoli1}, we get for $\rho \leq \frac{R}{2}$
\begin{equation}\label{qrnorm}
\aligned
&\|g - g_\rho\|_{L^2(\D_\rho)} \leq c\,\rho \|D_z g\|_{L^2(\D_\rho)} \leq c\,\rho\left(\int_{\D_\rho}K\;J(z, g)\right)^\frac12\\
&\quad\leq  c(K)\,\frac{\rho^{1 + 1/K}}{R^{1/K}}\left(\int_{\D_{R/2}}J(z, g)\right)^\frac12\leq c(K)\,\frac{\rho^{1 + 1/K}}{R^{1/K}}\|D_z g\|_{L^2(\D_{R/2})}\\
&\quad\leq c(K)\left(\frac{\rho}{R}\right)^{1 + 1/K}\|g - g_R\|_{L^2(\D_R)}.
\endaligned
\end{equation}
For $\frac{R}{2} < \rho \leq R$, \eqref{qrnorm} holds trivially. Hence  we have shown the integral estimate \eqref{normbound}.
\end{proof}

 The formulation of  Proposition~\ref{integralestimate}  will be particularly useful when applied to the derivatives $D_zF$ of a solution to the autonomous equation \eqref{auto}.

\begin{cor}\label{holder1K}
 If $F$ is as in Proposition~\ref{ccc}, the derivative $D_z F$ is  locally $\frac1K$-H\"older continuous. Moreover,
\begin{enumerate}
\item for every $\D(z_0, \rho) \subset \D(z_0, R) \subset \Omega$,
$$
\| D_z F \|_{L^2(\D(z_0, \rho))} \leq c(K)\,\frac{\rho}{R}\;\| D_z F \|_{L^2(\D(z_0, R))}.
$$
\item For every $\D(z_0, \rho) \subset \D(z_0, R) \subset \Omega$,
 $$
 \|D_z F-(D_z F)_\rho\|_{L^2(\D(z_0,\rho))} \leq  c(K)\left(\frac\rho{R}\right)^{1 + \frac1K}\,\|D_z F-(D_z F)_{R}\|_{L^2(\D(z_0,R))}
$$
 where $(D_z F)_r = \fint_{\D(z_0, r)} D_z F$.
\end{enumerate}
\end{cor} 

\begin{proof}
Since for the Hilbert-Schmidt norm $\| D_z F(z) \|^2 = \sum_{j=1}^2 |D_z F(z)\,e_j |^2 =  \sum_{j=1}^2 | \partial_{e_j} F(z)|^2$ and $\| D_z F - (D_z F)_r \|^2 = \sum_{j=1}^2 | \partial_{e_j} F - (\partial_{e_j} F)_r|^2$, the corollary follows from Proposition~\ref{integralestimate}  and the quasiregularity of the directional derivatives.
\end{proof}

\subsection{Riemann-Hilbert problem}

The solution of the Riemann-Hilbert problem is well-known; the proof is based on the local versions of the classical Cauchy transform and the Beurling transform, see, for instance, \cite[Proposition~2]{tartar}. We sketch a proof for the reader's convenience in the situation we need in this paper.

\begin{prop}\label{splitting}
Let  $f$ be a solution to the nonlinear Beltrami equation \eqref{hqrjac}, and suppose $\D(z_0, R) \Subset \Omega$. Then there exists a unique solution $F \in W^{1, 2}(\D(z_0, R), \C)$ to the following local Riemann-Hilbert problem for the autonomous equation
\begin{equation}\label{Split}
\begin{cases}
\partial_{\zbar} F(z) = \cH(z_0,\partial_z F(z))  & \text{a.e. $z \in \D(z_0, R)$}, \\
\Re(f - F) = 0 & \text{on $\partial \D(z_0, R)$}.
\end{cases}
\end{equation}
Furthermore, $\| \partial_{\zbar}F - \partial_{\zbar} f\|_{L^2(\D_R)} = \| \partial_{z}F - \partial_{z} f\|_{L^2(\D_R)}$ and we have a norm bound
\begin{equation}\label{L2bound}
\|D_z F \|_{L^2(\D_R)} \leq 2 K \|D_z f \|_{L^2(\D_R)}.
\end{equation}
\end{prop}

\begin{proof}
The local Cauchy transform in $\D_R:= \D(z_0, R)$ is obtained from the Cauchy transform on the unit disk by conformal change of variables (see, e.g., \cite[Section 6.1]{G-cl}). Namely, the local Cauchy transform of $\psi \in L^2(\D_R, \C)$ is given by
$$
(\cC_{\D_R} \psi)(z) = \frac1{\pi} \int_{\Omega} \left(\frac{\psi(\zeta)}{z - \zeta} -\frac{(z - z_0)\,\overline{\psi(\zeta)}}{R^2 - (z - z_0)\, \overline{(\zeta - z_0)}} \right)dA(\zeta)
$$
and the local Beurling transform by $\cS_{\D_R} \psi = \partial_z\, \cC_{\D_R} \psi$, that is,
$$
(\cS_{\D_R} \psi)(z) = -\frac{1}{\pi} \int_{\Omega} \left(\frac{\psi(\zeta)}{(z - \zeta)^2} +\frac{R^2\,\overline{\psi(\zeta)}}{(R^2 - (z - z_0)\, \overline{(\zeta - z_0)})^2} \right)dA(\zeta).
$$
By definition, $\partial_z\, \cC_{\D_R} \psi = \cS_{\D_R} \psi$, $\partial_{\zbar}\, \cC_{\D_R} \psi =  \psi$, and $\cC_{\D_R} \psi \in W^{1, 2}(\D_R, \C) \cap C(\overline{\D_R}, \C)$. 

As the integrand in the definition of $\cC_{\D_R}$ is purely imaginary on the boundary, $\Re(\cC_{\D_R} \psi) = 0$ on $\partial\D_R$, i.e., $\Re(\cC_{\D_R} \psi)$ is in the closure of $C^\infty_0(\D_R, \C)$ in $W^{1,2}(\D_R, \C)$.  Now,  we can use Green's theorem, \cite[Theorem~2.9.1]{AIM}, to see that the local Beurling transform $\cS_{\D_R} : L^2(\D_R, \C) \to L^2(\D_R, \C)$ is an isometry, that is,
$$
\|\cS_{\D_R}\psi\|_{L^2(\D_R)} = \|\psi\|_{L^2(\D_R)}.
$$
Indeed, let $\cC_{\D_R} \psi = u + iv$,
$$
\aligned
\int_{\D_R} |\cS_{\D_R}\psi|^2 - |\psi|^2 &= \int_{\D_R} |\partial_z\, \cC_{\D_R} \psi|^2 - |\partial_{\zbar}\, \cC_{\D_R} \psi|^2 = \int_{\D_R} J(z, \cC_{\D_R} \psi)\\
&= -\frac{i}{2}\int_{\D_R} \partial_z u\; \partial_{\zbar}v - \partial_{\zbar} u\; \partial_z v = \frac{1}{4}\int_{\partial\D_R} u(\partial_z v + \partial_{\zbar}v) = 0,
\endaligned
$$
as $u = 0$ on $\partial\D_R$.

The isometry of $\cS_{\D_R}$ implies that the Beltrami operator
$$
(\cB\psi)(z) = \cH(z_0,(\cS_{\D_R}\psi)(z) + \partial_z f(z)) - \cH(z,\partial_z f(z)) 
$$
is a contraction on $L^2(\D_R, \C)$;
$$
\aligned
&\|\cB\psi_1 - \cB\psi_2\|_{L^2(\D_R)} \\& \ \ \ = \|\cH(z_0,(\cS_{\D_R}\psi_1)(z) + \partial_z f(z)) - \cH(z_0,(\cS_{\D_R}\psi_2)(z) + \partial_z f(z))\|_{L^2(\D_R)}\\
&\ \ \ \leq k \|\psi_1 - \psi_2\|_{L^2(\D_R)}.
\endaligned
$$
Thus there is a unique fixed point $\Psi\in L^2(\D_R, \C)$ of $\cB$.

We define $F = \cC_{\D_R}\Psi + f$, since then $\partial_{\zbar}F = \Psi + \partial_{\zbar} f$, $\partial_z F = \cS_{\D_R}\Psi + \partial_z f$, and $\Re F = \Re (\cC_{\D_R}\Psi) + \Re f = \Re f$ (i.e., $F$ solves \eqref{Split}).

The $L^2$-estimate is obtained in the similar fashion. For the fixed point $\Psi$
$$
\aligned
\| \Psi \|_{L^2(\D_R)} &= \|\cB\Psi\|_{L^2(\D_R)}\\
&= \|\cH(z_0,(\cS_{\D_R}\Psi)(z) + \partial_z f(z)) - \cH(z, \partial_z f(z)) 
\|_{L^2(\D_R)}\\
&\leq k \|\cS_{\D_R}\Psi\|_{L^2(\D_R)} + 2k \|\partial_z f\|_{L^2(\D_R)}.
\endaligned
$$
Now, using that $\cS_{\D_R}$ is also an isometry,
$$
\aligned
\| D_z F \|_{L^2(\D_R)} &\leq \| \Psi \|_{L^2(\D_R)} + \| \cS_{\D_R}\Psi \|_{L^2(\D_R)} + 2\| D_z f \|_{L^2(\D_R)} \\
& \leq \left(\frac{4k}{1- k} + 2\right)\|D_z f\|_{L^2(\D_R)} = 2\;\frac{1+k}{1- k}\|D_z f\|_{L^2(\D_R)}.\endaligned
$$
\end{proof}

\subsection{Schauder estimates by freezing the coefficients}\label{schaudersec}

We will use the Morrey-Campanato integral characterisation of  H\"older continuous functions
\cite[Chapter III, Theorem 1.2, p. 70, and Theorem 1.3, p. 72]{Gia}. Namely,  the integral estimate
\begin{equation}\label{morreycamp}
\|g - g_{\rho} \|_{L^2(\D(z_0, \rho))} \leq M\,\rho^{1 + \gamma}
\end{equation}
for $z_0 \in \Omega$ and every $\rho \leq \min\{R_0, \mathrm{dist}(z_0, \partial\Omega)\}$ (for some $R_0$) gives the local $\gamma$-H\"older continuity of $g$ in $\Omega$. Moreover, for $\tilde{\Omega}\Subset\Omega$, \eqref{morreycamp} implies the H\"older seminorm bound
\begin{equation}\label{semi}
[g]_{C^{\gamma}(\tilde{\Omega})} \leq c(\gamma, \tilde{\Omega})\,M
\end{equation}
and the $L^\infty$-bound
\begin{equation}\label{sup}
\| g \|_{L^\infty(\tilde{\Omega})} \leq c(\gamma, \tilde{\Omega})\left(M\,\mathrm{diam}(\Omega)^\gamma + \|g\|_{L^2(\Omega)}\right)\!,
\end{equation}
see the proofs of Proposition 1.2 and Theorem 1.2 on pages 68--72 of \cite[Chapter III]{Gia}.

Next, we apply the ideas of freezing the coefficients to get few basic estimates for solutions to \eqref{hqrjac}. We start with the following

\begin{lem} \label{basicII}
Suppose $\cH$ satisfies the conditions \eqref{Hcondition} and let $f \in W^{1, 2}_{\loc}(\Omega, \C)$ be a solution to 
\begin{equation*}
\partial_{\zbar} f(z) = \cH(z, \partial_z f(z))  \qquad \text{a.e. \, in} \;\;  \Omega.
\end{equation*}
If $\D(z_0, R) \Subset \Omega$, then for each $0 < \rho \leq R$ we have
$$
\aligned \| D_z f - (D_zf )_\rho \|_{L^2(\D_\rho)}  &\leq  c(K) \left(\frac{\rho}{R}\right)^{1+\frac1K} \|D_z f - (D_zf )_R\|_{L^2(\D_R)}\\
&\quad + c(K) \,  \mathbf{H}_\alpha(\Omega)\, R^\alpha\, \| \partial_z f \|_{L^2(\D_R)}, 
\endaligned
$$
where $\D_r = \D(z_0, r)$.
\end{lem}

\begin{proof} The required estimate to prove is then the same as  in Corollary \ref{holder1K}, claim $(2)$, up to the correction term   
$c(K)\,   \mathbf{H}_\alpha(\Omega) R^\alpha \| \partial_zf \|_{L^2(\D_R)} $. This will  arise from a comparison of $f$ and the solution $F$ to an autonomous equation, the local Riemann-Hilbert problem
$$
\begin{cases}
\partial_{\zbar} F(z) = \cH(z_0, \partial_z F(z))  & \text{a.e. $z \in \D_R$}, \\
\Re(f - F) = 0 & \text{on $\partial \D_R$}.
\end{cases}
$$
The existence of $F$  follows by Proposition~\ref{splitting}. Furthermore, by \eqref{Hcondition}, \\
$$
\aligned
&\|\partial_{\zbar} (f-F)\|_{L^2(\D_R)} \\
&\qquad \leq \|\cH(z,\partial_z f)-\cH(z_0,\partial_z f)\|_{L^2(\D_R)} + \|\cH(z_0, \partial_z f)-\cH(z_0,\partial_z F)\|_{L^2(\D_R)} \\
&\qquad\leq2\,\mathbf{H}_\alpha(\Omega)\,R^\alpha\,\|\partial_z f\|_{L^2(\D_R)} +k\, \|\partial_z (f- F)\|_{L^2(\D_R)}.
\endaligned
$$\\
Since the Beurling transform  $\cS_{\D_R}$ of the disk $\D_R$ is an isometry $L^2(\D_R) \to L^2(\D_R)$, we end up with
\begin{equation} \label{basicI}
\| D_z f - D_z F  \|_{L^2(\D_R)}  \leq \frac4{1 - k}\,  \mathbf{H}_\alpha(\Omega)\,R^\alpha\,\|\partial_z f\|_{L^2(\D_R)}.
\end{equation}
\vspace{-.3cm}

On the other hand, Corollary  \ref{holder1K} $(2)$ gives 
$$
\aligned
&\| D_z f - (D_zf )_\rho \|_{L^2(\D_\rho)}  \leq \| D_z F - (D_zF )_\rho \|_{L^2(\D_\rho)} + 2 \| D_z f - D_z F  \|_{L^2(\D_\rho)}\\
&\quad\leq c(K)\left(\frac\rho{R}\right)^{1 + \frac1K}\,\|D_z F-(D_z F)_{R}\|_{L^2(\D_R)} + 2 \| D_z f - D_z F  \|_{L^2(\D_R)}\\
&\quad\leq c(K)\left(\frac\rho{R}\right)^{1 + \frac1K}\,\|D_z f-(D_z f)_{R}\|_{L^2(\D_R)} + (2\,c(K) + 2) \| D_z f - D_z F  \|_{L^2(\D_R)},
\endaligned
$$
$\rho \leq R$. Combining this  with \eqref{basicI} gives the claim.
\end{proof}

If we use  claim $(1)$ of Corollary  \ref{holder1K}, instead of  claim $(2)$,  the same argument as above leads to 
\begin{lem} \label{basicIV}
Suppose $\cH$ satisfies the conditions \eqref{Hcondition}. If $f \in W^{1, 2}_{\loc}(\Omega, \C)$ and   $\D(z_0, R)$ are as in Lemma
\ref{basicII}, then for each $0 < \rho \leq R$,
$$ \| D_z f \|_{L^2(\D_\rho)} \leq c(K)\,  \frac{\rho}{R}\,  \| D_z f \|_{L^2(\D_R)} + c(K) \,  \mathbf{H}_\alpha(\Omega)\, R^\alpha\, \| \partial_z f \|_{L^2(\D_R)}.
$$
\end{lem}
\smallskip

Since the $W^{1,2}_{\loc}$-solutions  to \eqref{hqrjac} are a priori $K$-quasiregular, we have the Caccioppoli estimates \eqref{caccioppoli1} at our use. These are convenient to present in the following form.

\begin{lem} \label{cacciopp2}
 Suppose $\cH$ and $f \in W^{1, 2}_{\loc}(\Omega, \C)$  are as in Lemma \ref{basicII}. Let $\D(z_0, R) \subset \Omega'' \Subset \Omega' \Subset \Omega$. If $f\in C^\beta (\Omega', \C)$ for some $0 < \beta \leq 1$, then
$$ \| D_z f \|_{L^2(\D(z_0, R))} \leq c(K, \Omega', \Omega'') \, [f]_{C^\beta(\Omega')} \, R^\beta.$$
\end{lem}

Lastly, let us  recall
\begin{lem}[Lemma 2.1, p. 86, in  {\cite[Chapter III]{Gia}}]\label{growth}
Let $\Psi$ be non-negative, non-decreasing function such that
$$
\Psi(\rho) \leq a\left[\left(\frac{\rho}{R}\right)^{\lambda} + \sigma\right]\Psi(R) + bR^{\gamma}
$$
for every $0 < \rho \leq R \leq R_0$, where $a$ is non-negative constant and $0 < \gamma < \lambda$. Then there exists $\sigma_0 = \sigma_0(a, \lambda, \gamma)$ such that, if $\sigma < \sigma_0$,
$$
\Psi(\rho) \leq c(a, \lambda, \gamma)\left[\left(\frac{\rho}{R}\right)^{\gamma}\Psi(R) + b\rho^{\gamma}\right]
$$
for all $0 < \rho \leq R \leq R_0$.
\end{lem}

\medskip

With these tools and estimates at our disposal we are ready for   the Schauder estimates.

\begin{proof}[Proof of Theorem~\ref{schauder}] Denote $\D_r = \D(z_0, r)$.
\medskip

\noindent{\em Step 1. H\"older continuity of $f$.}\quad We will show that $f$ is actually locally $\beta$-H\"older continuous for every $0 < \beta < 1$. 

Namely, according to Lemma \ref{basicIV} we have
\begin{equation*} \label{guiqintaI} 
\| D_z f \|_{L^2(\D_\rho)} \leq c_0(K) \left( \frac{\rho}{R} \, +  \,  \mathbf{H}_\alpha(\Omega) R^\alpha \right)  \| D_z f \|_{L^2(\D_R)},
\end{equation*}
whenever $0 < \rho \leq R$ and $\D_R = \D(z_0, R) \subset \Omega$.  Applying Lemma \ref{growth} to $\Psi(\rho) = \| D_z f \|_{L^2(\D_\rho)}$, with $b=0, \lambda = 1$ and $\sigma= \mathbf{H}_\alpha(\Omega) R^\alpha$, we see that 
$$\| D_z f \|_{L^2(\D_\rho)} \leq c_1(K)  \left( \frac{\rho}{R} \right)^{1-\epsilon}  \| D_z f \|_{L^2(\D_{R})}, 
$$
where $0 < \rho \leq R \leq \min \{R_0, \dist(z_0, \partial \Omega)\}$. 
Here $R_0$ is small enough; how small $R_0$ needs to be taken depends  on $c_0(K),  \mathbf{H}_\alpha(\Omega)$ and $\epsilon >0$ but not on $z_0$. Thus  the same upper bound  $R_0$ works throughout the bounded domain $\Omega$. 

Combining with the Poincar\'e inequality gives
$$ \|f-f_\rho\|_{L^2(\D_\rho)}  \leq  \rho\, \| D_z f \|_{L^2(\D_\rho)} \leq c_1(K)\, \rho^{2-\epsilon} \,R^{\epsilon-1}\, \| D_z f \|_{L^2(\D_R)},
$$
for $0 < \rho \leq R \leq \min \{R_0, \dist(z_0, \partial \Omega)\}$. 

Let $\D(\omega, 4R) \subset \Omega$. Now, for $\D(z_0, \rho) \subset \D(\omega, 2R)$,
$$ \|f-f_\rho\|_{L^2(\D_\rho)}   \leq c_1(K)\, \rho^{2-\epsilon} \, \min\{R_0, R \}^{-\beta} \| D_z f \|_{L^2(\D(\omega, 3R))}.
$$
In view of \eqref{morreycamp} we see that $f \in C^\beta_{\loc}(\D(\omega, 2R),\C)$ for every $0 < \beta < 1$.  The estimate \eqref{semi} gives a bound for the local H\"older norm,
\begin{equation}\label{fholdernorm}
 [f]_{C^\beta(\D(\omega, R))} \leq c_2(K, \beta, R, \mathbf{H}_\alpha(\Omega)) \, \| D_z f \|_{L^2(\D(\omega, 3R))}.
 \end{equation}

\medskip

\noindent{\em Step 2: Self-improving Morrey-Campanato estimate.}\quad 
Claim: Assume that $1<\alpha+\beta<1+\frac1K$. Then  $D_z f\in C^{\alpha+\beta-1}_{\loc}(\Omega, \C)$. 

\smallskip

Let $\Omega'' \Subset \Omega' \Subset \Omega$. We first show the claim for $\beta < 1$, and start with estimates from Lemma \ref{basicII}, 
$$
\aligned
\| D_z f - (D_zf )_\rho \|_{L^2(\D_\rho)}  &\leq  c_0(K) \left(\frac{\rho}{R}\right)^{1+\frac1K} \|D_z f - (D_zf )_R\|_{L^2(\D_R)}\\ 
&\quad+ c_0(K) \,  \mathbf{H}_\alpha(\Omega)\, R^\alpha\, \| \partial_z f \|_{L^2(\D_R)},
\endaligned $$
when $ \D(z_0, R) \subset \Omega''$. Here, by the Caccioppoli estimate of Lemma \ref{cacciopp2} 
\begin{equation}
\label{boundedness}
 \| \partial_z f \|_{L^2(\D_R)} \leq c_1(K, \Omega', \Omega'') \, [f]_{C^\beta(\Omega')} \, R^\beta, 
\end{equation} 
which by Step 1 is finite for every $\beta < 1$.

We will now apply Lemma \ref{growth} to the non-decreasing function $\Psi(\rho) = \|D_z f-(D_z f)_\rho\|_{L^2(\D_\rho)} = \inf_{a\in \C} \| D_z f - a \|_{L^2(\D_\rho)}$ and the parameters  $\lambda = 1+ \frac1{K}$, $\sigma = 0$ and $b = \mathbf{H}_\alpha(\Omega)   \,[f]_{C^\beta(\Omega')} $. We obtain that
\begin{equation}\label{gammanorm} 
\aligned
\| D_z f - (D_zf )_\rho \|_{L^2(\D_\rho)}  &\leq  c_2 \left(\frac{\rho}{R}\right)^{\alpha+\beta}  \|D_z f - (D_zf )_R\|_{L^2(\D_R)} \\
&\quad + c_2 \, \rho^{\alpha+\beta} \,  \mathbf{H}_\alpha(\Omega)   \,[f]_{C^\beta(\Omega')} 
\endaligned
\end{equation}
whenever $\rho\leq R$. 

In terms of the Morrey-Campanato estimate \eqref{morreycamp} in the set $\Omega''$, we see that 
$D_z f\in C^{\alpha+\beta-1}_{\loc}(\Omega'', \C)$, which is enough for our claim if $\alpha \geq 1/K$. The norm estimate \eqref{thmnorm} follows  
from combining \eqref{semi} with  \eqref{fholdernorm} and \eqref{gammanorm}.

In case $\alpha < 1/K$ we need to continue to show that $f \in C^{1, \alpha}_{\loc}(\Omega, \C)$. But  what we have  above proves that 
 $D_z f$ is locally bounded. Thus the bound  in \eqref{boundedness} remains finite for $\beta = 1$, and we can repeat the proof of \eqref{gammanorm} with $\beta = 1$. Accordingly, \eqref{morreycamp} and \eqref{semi}  give $f \in C^{1, \alpha}_{\loc}(\Omega, \C)$, with norm bound

 $$
 \aligned
{[D_zf]}_{C^\alpha(\D(\omega, R))}
&\leq c(K, \alpha, \omega, R)\biggl[\|D_z f\|_{L^2(\D(\omega, 2R))}\\
&\qquad    + \mathbf{H}_\alpha(\Omega)\,\|D_z f\|_{L^\infty(\D(\omega, 2R))} \biggr].\endaligned
$$

To estimate the $L^\infty$-norm in $\D(\omega, 2R)$, we note that for $\D(z_0, \rho) \subset \D(\omega, \frac{5R}{2})$ \eqref{gammanorm} holds with $\Omega' = \D(\omega, 3R)$ and thus once more by Morrey-Campanato norm estimate \eqref{morreycamp} (with \eqref{sup})
$$
\aligned
\|D_z f\|_{L^\infty(\D(\omega, 2R))} &\leq c(K, \alpha, \omega, R)\biggl[\|D_z f \|_{L^2(\D(\omega, 3R))}\\
&\qquad + \mathbf{H}_\alpha(\Omega)\, [f]_{C^{\beta'}(\D(\omega, 3R))}\biggr],
\endaligned
$$
where $\beta'< 1$. It remains to combine with  \eqref{fholdernorm} to obtain

 $$
{\|D_zf\|}_{C^\gamma(\D(\omega, R))}
\leq c(K, \alpha, \gamma, \omega, R, \mathbf{H}_\alpha(\Omega))\,\|D_z f\|_{L^2(\D(\omega, 9R))},
$$
and we have  the norm bound  \eqref{thmnorm} by rescaling.
\end{proof}

\subsection{Schauder estimates with $C^1$ gradient dependence}\label{C1sec}

\begin{proof}[Proof of Theorem \ref{schauderC1}]
As we see in Step 2 of the proof of Theorem~\ref{schauder}, the restriction on H\"older continuity comes from the autonomous case. Hence it is enough to show that, when the dependence on the gradient is $C^1$, we may improve the norm estimates in Corollary~\ref{holder1K}.

\begin{prop}\label{ccc1}
Let $F \in W^{1,2}_{\loc}(\Omega, \C)$ be a solution to the autonomous nonlinear Beltrami equation  \eqref{auto}, where in addition $\xi \mapsto \cH(\xi) \in C^1(\C, \C)$. Then,
for every $\epsilon > 0$ and $\D(z_0, \rho) \subset \D(z_0, R) \subset \Omega'' \Subset \Omega' \Subset \Omega$,
 $$
 \|D_z F-(D_z F)_\rho\|_{L^2(\D(z_0,\rho))} \leq  c\,\left(\frac\rho{R}\right)^{2 - \epsilon}\,\|D_z F-(D_z F)_{R}\|_{L^2(\D(z_0,R))}
$$
 where $(D_z F)_r = \fint_{\D(z_0, r)} D_z F$ and the constant $c$ depends on the parameters $K$, $\Omega'$, $\Omega''$, $\|D f\|_{L^2(\Omega')}$ and the modulus of continuity of $\cH_{\xi}$ and $\cH_{\bar{\xi}}$. 
\end{prop}

\begin{proof} 
We know by Proposition~\ref{ccc} that $\partial_z  F(z) \in W^{1, 2}_{\loc}(\Omega, \C)$. If we differentiate the autonomous equation $\partial_{\zbar} F(z) = \cH(\partial_z F(z))$ with respect to $z$, we get for $g = \partial_z F$ that 
\begin{equation*}
g_{\zbar} =  \cH_\xi(g)\,\partial_z g + \cH_{\bar{\xi}}(g) \,\overline{\partial_{\zbar} g}, \qquad \text{a.e. \quad in $\Omega$}.
\end{equation*}
By isolating $g_{\zbar}$ we obtain the $\R$-linear Beltrami equation
\begin{equation}\label{automunu2}
g_{\zbar} = \mu(g)\,g_z + \nu(g)\,\bar{g_z}
\end{equation}
with the coefficients
\begin{equation}\label{automunu}
\mu(g) = \frac{\cH_\xi(g)}{1 - |\cH_{\bar{\xi}} (g)|^2}, \quad \nu(g) = \frac{\bar{\cH_\xi(g)}\,\cH_{\bar{\xi}}(g)}{1 - |\cH_{\bar{\xi}} (g)|^2},
\end{equation}
satisfying
$$
|\mu(g)| + |\nu(g)| \leq k < 1,
$$
by $k$-Lipschitz property of $\cH$, $|D_\xi \cH(g)| = |\cH_\xi(g)| + |\cH_{\bar{\xi}}(g)| \leq k < 1$. 

\medskip

There are now two natural ways to proceed. First, we have a quick  way to  deduce the $W^{2, p}_{\loc}$-regularity of the solution $F$ for all $1 < p < \infty$ using \eqref{automunu2} and the fact that the  coefficients $\mu$, $\nu$  are continuous. In fact, following the path from \cite{AIS,clop,koski},  for any linear Beltrami equation with coefficients  in $ VMO$ all  $W^{1,2}_{\loc}$-solutions are actually $W^{1,p}_{\loc}$-regular for every $1 < p < \infty$.  However, these arguments 
rely on applying Fredholm theory to the Beltrami equation and as such do not yield  the explicit bounds we need in a straightforward manner.
\medskip

Another approach is to use the Morrey-Campanato method to  improve the norm estimates in Corollary~\ref{holder1K}. Here we split $g = G + (g - G)$, where $G$ solves the Riemann-Hilbert problem of a linear equation with constant coefficients, 
\begin{equation}\label{g}
\begin{cases}
 G_{\zbar} = \mu((\partial_z F)_R)\, G_z + \nu((\partial_z F)_R)\, \bar{G_{z}} & \text{a.e. $z \in\D_R = \D(z_1, R)$}, \\
\Re(g - G) = 0 & \text{on $\partial \D_R$}.
\end{cases}
\end{equation}
Above $z_1 \in \Omega''$, $R \leq \mathrm{dist}(z_1, \partial \Omega'')$, and $(\partial_z F)_R =  \fint_{\D_R} \partial_z F$. Similarly, as we already saw in Proposition~\ref{splitting}, the existence of $G$ is  based on the local versions of the classical Cauchy transform and the Beurling transform. Moreover, 
\begin{equation}\label{gboundforG}
\| D_z G \|_{L^2(\D_R)} \leq c(K) \|D_z g\|_{L^2(\D_R)}.
\end{equation}

For a solution $G$ to \eqref{g}, we have $D_zG  \in W^{1, 2}_{\loc}(\D_R, \C)$ and the directional derivatives $\partial_e  G(z)$, $|e| = 1$, are $K$-quasiregular in $\D_R$, by using difference quotients $G_h$ as in the proof of Proposition~\ref{ccc}. 

We will show that, for every $\epsilon > 0$,
\begin{equation}\label{goal}
\|D^2_z F\|_{L^2(\D_\rho)} \leq  c(K, \epsilon)\left(\frac\rho{R}\right)^{1 - \epsilon}\,\|D^2_z F\|_{L^2(\D_R)},
\end{equation}
whenever $\rho \leq R \leq \min\{ R_0, \mathrm{dist}(z_1, \partial\Omega'')\}$, where $R_0$ will be chosen later.

As $\partial_{\zbar} F = \cH(\partial_z F)$ and $D_\xi \cH(\xi)$ is uniformly bounded by $k$, it is enough to show the claim for $g = \partial_z F$, that is,
\begin{equation}\label{goal2}
\|D_z g\|_{L^2(\D_\rho)} \leq  c(K, \epsilon)\left(\frac\rho{R}\right)^{1 - \epsilon}\,\|D_z g\|_{L^2(\D_R)},
\end{equation}
whenever $\rho \leq R \leq \min\{ R_0, \mathrm{dist}(z_1, \partial\Omega'')\}$.

Since  $\| D_z G(z) \|^2 = \sum_{j=1}^2 |D_z G(z)\,e_j |^2 =  \sum_{j=1}^2 | \partial_{e_j} G(z)|^2$ for the Hilbert-Schmidt norm, the quasiregularity of $\partial_e G$ with integral estimate \eqref{g1} of Proposition~\ref{integralestimate} implies
$$
\|D_z G \|_{L^2(\D_\rho)} \leq c(K)\,\frac{\rho}{R}\;\| D_z G \|_{L^2(\D_R)}.
$$
Hence, by triangle inequality,
\begin{equation}\label{firststage}
\aligned
\|D_z g\|_{L^2(\D_\rho)} &\leq \|D_z G\|_{L^2(\D_\rho)} + \|D_z (g - G)\|_{L^2(\D_\rho)}\\
&\leq c(K)\,\frac{\rho}{R}\;\| D_z G \|_{L^2(\D_R)} + \|D_z (g - G)\|_{L^2(\D_R)}\\
&\leq c(K)\,\frac{\rho}{R}\;\| D_z g \|_{L^2(\D_R)} + \|D_z (g - G)\|_{L^2(\D_R)},
\endaligned
\end{equation}
where the last estimate follows by \eqref{gboundforG}.
Thus we need to estimate  $D_z (g - G)$. Below we use the uniform bound $|\mu| + |\nu| \leq k$ to get that
\begin{align*}
\|(g - &G)_{\zbar}\|_{L^2(\D_R)}\\
&=  \|\mu(g)\, g_z + \nu(g)\, \bar{g_z} - \mu((\partial_z F)_R)\, G_z - \nu((\partial_z F)_R)\, \bar{G_z}\|_{L^2(\D_R)} \\
&\quad\leq  \|\mu((\partial_z F)_R)\,(g - G)_z + \nu((\partial_z F)_R)\,\bar{(g- G)_z}\|_{L^2(\D_R)} \\
&\qquad+  \|(\mu(g) - \mu((\partial_z F)_R))\, g_z + (\nu(g) - \nu((\partial_z F)_R))\, \bar{g_z}\|_{L^2(\D_R)} \\
&\quad\leq k\, \|(g - G)_{z}\|_{L^2(\D_R)}\\
&\qquad + \sup_{z\in \D_R}\bigl[ |(\mu(g) - \mu((\partial_z F)_R)| + |\nu(g) - \nu((\partial_z F)_R))|\bigr] \|g_z\|_{L^2(\D_R)}.
\end{align*}

Hence, combining with \eqref{firststage} and using that the local Beurling transform of the disk is an isometry to absorb the term $k\, \|(g - G)_{z}\|_{L^2(\D_R)}$ into the left hand side, we have
$$
\|D_z g\|_{L^2(\D_\rho)}  \leq c(K)\left(\frac{\rho}{R} + \sigma(R)\right) \|D_zg\|_{L^2(\D_{R})},
$$ 
where
$$
\sigma(R) := \sup_{z\in \D_R}\bigl[ |(\mu(\partial_z F) - \mu((\partial_z F)_R)| + |\nu(\partial_z F) - \nu((\partial_z F)_R))|\bigr].
$$

Now, \eqref{goal2} follows by Lemma~\ref{growth} if we can make $\sigma(R)$ as small as we wish by reducing $R$. This is actually possible since $\partial_z F$ is $\frac1K$-H\"older continuous by Corollary~\ref{holder1K} and $\mu$ and $\nu$ are continuous by the fact that $\cH$ is $C^1$. Here $R_0$ has to be so small that $\sigma(R_0) \leq \sigma_0(K, \epsilon)$, where the constant $\sigma_0$ is from Lemma~\ref{growth}. Moreover, we can  choose $R_0$ uniformly in the compact set $\Omega''$. 

We collect now the dependence of $R_0$ on the parameters. From the proof we see that it depends on the modulus of continuity of $\cH_{\xi}$ and $\cH_{\bar{\xi}}$ on the set $\partial_z F(\Omega'')$ as well as the numbers $[\partial_z F]_{C^{1/K}(\Omega'')}$, $K$ and $\epsilon$. It is also possible, via Corollary~\ref{holder1K} and the Morrey-Campanato norm estimates \eqref{morreycamp}--\eqref{sup}, to bound the size of the set $\partial_z F(\Omega'')$  and $[\partial_z F]_{C^{1/K}(\Omega'')}$ in terms of $c(K, \Omega', \Omega'')\,\|D F\|_{L^2(\Omega')}$.

Using Poincar\'e's inequality on the left hand side of \eqref{goal} and Caccioppoli's inequality on the right we deduce, for  $\rho \leq R \leq \min\{R_0, \mathrm{dist}(z_1, \partial \Omega'')\}$,
\begin{equation}\label{Campanato}
\begin{array}{l}
\|D_z F - (D_z F)_\rho\|_{L^2(\D_\rho)} \\
\qquad \qquad \leq c(K, \epsilon)\left(\frac{\rho}{R}\right)^{1 + (1-\epsilon)}\|D_z F - (D_z F)_R\|_{L^2(\D_R)}.
\end{array}
\end{equation}
As we have seen before, because of the Caccioppoli estimate, we have \eqref{Campanato} first for $\rho \leq \frac{R}{2}$. The full range $\rho \leq R$ holds with possible bigger constant.

The claim follows by covering $\D(z_0, R)$ with disks of radius $R_0$.
\end{proof}

The (nonlinear) $L^2$-norm dependence of ${R}_0$ is reflected in the final $C^\alpha$-norm estimate of $D_z f$, that is, we do not have linear dependence on the $L^2$-norm as in Theorem~\ref{schauder} (i.e., \eqref{thmnorm}).
\end{proof}

\section{Positivity of the Jacobian}\label{Jacproof}

\noindent In this section we prove Theorem~\ref{Jac}. The proof is based on Schauder estimates and the following lemma that establishes a nonlinear equation for the inverses of our solutions. After the proof of Theorem~\ref{Jac} we will give a simple and separate argument for the autonomous case.

\begin{lem}\label{Hinverse} Let $f$ be a homeomorphic solution to the nonlinear Beltrami equation
\[\partial_{\bz}f(z) = \hhh(z,\partial_z f(z)) \qquad \text{a.e. \quad in } \Omega,\]
where $\hhh$ is measurable in $z$ and $k$-Lipschitz, $k = \frac{K-1}{K+1}$, in the gradient variable with the normalisation $\hhh(z,0) \equiv 0$. Then $g = f^{-1}$ solves a nonlinear equation of the form
\[\partial_{\bar{\omega}}g(\omega) = \hhh^*(g(\omega),\partial_{\omega}g(\omega)) \qquad \text{a.e. \quad in } f(\Omega).\]
Moreover,
\begin{itemize}
\item The function $\hhh^*$ does not depend on the solution $f$.
\item $\hhh^*(g,\xi)$ is measurable in the variable $g$ and $\frac{K^3-1}{K^3+1}$-Lipschitz in $\xi$ with the normalisation $\hhh^*(g,0) \equiv 0$.
\item If $\hhh$ satisfies the H\"older condition \eqref{Hcondition}, then so does $\hhh^*$.
\end{itemize}
\end{lem}

\begin{rem} Note that the assumption that $\hhh(z,\xi)$ is measurable in $z$ and $k$-Lipschitz in $\xi$ implies the natural condition of Lusin-measurability of $\hhh$, as defined in \cite[Section 7.6]{AIM}. Nevertheless, these assumptions are only in place to make the statement of Lemma~\ref{Hinverse} more general as throughout the rest of the paper we assume the condition \eqref{Hcondition}, in particular that $\hhh$ is continuous in $(z,\xi)$.
\end{rem}

\begin{proof} [Proof of Lemma~\ref{Hinverse}.] Note first that since $f$ is quasiconformal, we have $J_f(z) > 0$ almost everywhere in $\Omega$. The identities
\[\frac{-g_{\bar{\omega}}}{J_g} = f_{\bz} \quad \text{ and } \quad \frac{\bar{g_{\omega}}}{J_g} = f_z \qquad \text{at $\omega = f(z)$}\]
are also valid almost everywhere. Thus we find that $g$ satisfies the nonlinear equation
\begin{equation}\label{gequ}\frac{-g_{\bar{\omega}}}{|g_{\omega}|^2 - |g_{\bar{\omega}}|^2} = \hhh\left(g,\frac{\bar{g_{\omega}}}{|g_{\omega}|^2 - |g_{\bar{\omega}}|^2}\right)\end{equation}
almost everywhere in $f(\Omega)$. We first want to show that $g_{\bar{\omega}}$ can be uniquely solved from this equation in terms of $g$ and $g_\omega$. To do this we consider \eqref{gequ} as an equation of three complex variables:
\begin{equation}\label{EqSolvable}\frac{-\zeta}{|\xi|^2 - |\zeta|^2} =  \hhh\left(g,\frac{\bar{\xi}}{|\xi|^2 - |\zeta|^2}\right),\end{equation}
where the variables are as follows:
\begin{equation}\label{alue}
\left\{
\begin{array}{l}
\xi, \text{ an arbitrary complex number},  \\
\zeta, \text{ a complex variable that satisfies } |\zeta| \leq k|\xi|,\\
g, \text{ a complex variable that belongs to the set } \Omega.
\end{array} \right.
\end{equation}
We now solve \eqref{EqSolvable} in terms of $\zeta$. Fix the variables $g$ and $\xi$ and consider the function
\[F_{\xi}(\zeta) = \frac{-\zeta}{|\xi|^2 - |\zeta|^2}.\]
Then it is easy to check that $F_{\xi}$ is bijective from the disk $|\zeta| \leq k|\xi|$ onto another disk. We can hence make a substitution $\zeta = F_{\xi}^{-1}(\zeta')$ into \eqref{EqSolvable}. This transforms the equation into
\begin{equation}\label{zetaprime}\zeta' =  \hhh\left(g,\frac{\bar{\xi}}{|\xi|^2 - |F_{\xi}^{-1}(\zeta')|^2}\right).\end{equation}
Now, it happens that the map
\begin{equation*}\label{zetamap}\zeta' \mapsto \frac{\bar{\xi}}{|\xi|^2 - |F_{\xi}^{-1}(\zeta')|^2}\end{equation*}
is actually a contraction. This can be seen by differentiation, for example. Thus by the $k$-Lipschitz property of $\hhh$, the expression on the right hand side of \eqref{zetaprime} is a contraction in terms of $\zeta'$. Hence it has a unique fixed point, which shows that the equation can be uniquely solved for $\zeta'$. Thus $\zeta$ can also be uniquely solved from \eqref{EqSolvable} in terms of $g$ and $\xi$ in the disk $|\zeta| \leq k|\xi|$, and we can use this as the definition of $\hhh^*$,
\[\zeta = \hhh^*(g,\xi).\]
We would also like to make sure that the function $\hhh^*$ is measurable in the variable $g$. This follows from the fact that $\hhh^*(g,\xi)$ can be obtained by iterating the right hand side of \eqref{zetaprime}. At each point of the iteration the function is measurable due to Lusin-measurability of $\hhh$ and \cite[Theorem 7.7.2]{AIM}, and the limit function is measurable as a pointwise limit of measurable functions.

\smallskip

We now show that $\hhh^*$ is $\frac{K^3-1}{K^3+1}$-Lipschitz in the variable $\xi$. To do this we have to show that if
\[\zeta_j = \hhh^*(g,\xi_j) \quad j = 1,2,\]
then $|\zeta_1 - \zeta_2| \leq \frac{K^3-1}{K^3+1} |\xi_1-\xi_2|$. By definition $\zeta_j$ solves the equation \eqref{EqSolvable} for $g$ and $\xi_j$. Thus
\begin{equation}\label{zetajEq}\left|\frac{\zeta_2}{|\xi_2|^2 - |\zeta_2|^2} - \frac{\zeta_1}{|\xi_1|^2 - |\zeta_1|^2}\right| \leq \frac{K-1}{K+1}\left|\frac{\xi_2}{|\xi_2|^2 - |\zeta_2|^2} - \frac{\xi_1}{|\xi_1|^2 - |\zeta_1|^2} \right|.\end{equation}
Define two linear maps by $A_j z = \xi_j z + \zeta_j \bz$. We say that a linear map $A$ is $K$-quasiregular if
\[|A|^2 \leq K\det A.\]
Then the maps $A_j$ are $K$-quasiregular by the property $|\zeta_j| \leq \frac{K-1}{K+1}|\xi_j|$. By \eqref{zetajEq}, the map $A_2^{-1} - A_1^{-1}$ is also $K$-quasiregular. We can now use the identity
\[A_1 - A_2 = A_2\,(A_2^{-1} - A_1^{-1})\,A_1\]
to find that the linear map $A_1 - A_2$ is $K^3$-quasiregular. This gives that $|\zeta_1 - \zeta_2| \leq \frac{K^3-1}{K^3+1} |\xi_1-\xi_2|$ as wanted.

\smallskip

It remains to prove that the condition \eqref{Hcondition} is also inherited by $\hhh^*$. We must prove that there exists a constant $C$ such that if
\[\zeta_j = \hhh^*(g_j,\xi) \quad j = 1,2,\]
then $|\zeta_1 - \zeta_2| \leq C |g_1 - g_2|^{\alpha} |\xi|$. By using the property \eqref{Hcondition} of $\hhh$, we obtain that
\begin{align*}\left|\frac{\zeta_1}{|\xi|^2 - |\zeta_1|^2} - \frac{\zeta_2}{|\xi|^2 - |\zeta_2|^2}\right| &\leq
\mathbf{H}_{\alpha}(\Omega)\,|g_1 - g_2|^{\alpha}\left(\frac{|\xi|}{|\xi|^2 - |\zeta_1|^2}+\frac{|\xi|}{|\xi|^2 - |\zeta_2|^2}\right)
\\&\quad +\frac{K-1}{K+1}\left|\frac{\xi}{|\xi|^2 - |\zeta_1|^2} - \frac{\xi}{|\xi|^2 - |\zeta_2|^2}\right|.
\end{align*}
Denote $R = \frac{|\xi|^2 - |\zeta_1|^2}{|\xi|^2 - |\zeta_2|^2}$. Then the above estimate can also be written as
\[|\zeta_1 - R\,\zeta_2| - k|\xi||1-R| \leq \mathbf{H}_{\alpha}(\Omega)\,|g_1 - g_2|^{\alpha}|\xi|(1 + R).\]
Note that $R$ is bounded from above by a constant, thus the right hand side is already of the desired form. It remains to prove the elementary inequality
\[|\zeta_1 - R\,\zeta_2| - k|\xi||1-R| \geq c\,|\zeta_1 - \zeta_2|\]
for a sufficiently small constant $c > 0$. This will follow once we prove the two estimates
\begin{equation}\label{zetaestim1}
|\zeta_1 - R\,\zeta_2| \geq c\,|\zeta_1 - \zeta_2|
\end{equation}
and
\begin{equation}\label{zetaestim2}
|\zeta_1 - R\,\zeta_2| \geq |\xi||1-R|.
\end{equation}
To prove \eqref{zetaestim1}, it is enough to do the following estimate:
\begin{align*}|\zeta_1 - R\,\zeta_2| &= \left|\frac{(\zeta_1 - \zeta_2)|\xi|^2 + \zeta_1\zeta_2(\bar{\zeta_1} - \bar{\zeta_2})}{|\xi|^2 - |\zeta_2|^2}\right| \geq \frac{|\xi|^2 - k^2|\xi|^2}{|\xi|^2}|\zeta_1 - \zeta_2|.
\end{align*}
For \eqref{zetaestim2}, it suffices to estimate:
\begin{align*}
|\zeta_1 - R\,\zeta_2| &\geq ||\zeta_1| - R|\zeta_2||
\\&= \left||\zeta_1| - |\zeta_2|\right|\frac{|\xi|^2 + |\zeta_1||\zeta_2|}{|\xi|^2 - |\zeta_2|^2}
\\& \geq \left||\zeta_1| - |\zeta_2|\right|\frac{|\xi|(|\zeta_1|+|\zeta_2|)}{|\xi|^2 - |\zeta_2|^2}
\\&= |\xi||1-R|.
\end{align*}
This finishes the proof of Lemma~\ref{Hinverse}.
\end{proof}

\begin{proof} [Proof of Theorem~\ref{Jac}.] Combining Theorem~\ref{schauder} and Lemma~\ref{Hinverse}, we find that both $f$ and $g = f^{-1}$ are $C^{1,\alpha}_{\loc}$-regular. Outside of the set where $J_f(z) = 0$ we have the identity
\[J_f(z) J_g(f(z)) = 1.\]
Both of the Jacobians are continuous functions, which means that this identity must hold everywhere. This shows that $J_f(z) > 0$ everywhere.
\medskip

To complete the proof of Theorem \ref{Jac} we next use a compactness argument to show that for a normalised homeomorphic solution $f$ to \eqref{hqrjac} there is a lower bound for the Jacobian in each disk $\D(0, R_0)$, that is,
$$
\inf_{z\in\D(0, R_0)}J(z, f) \geq c(\cH, R_0) > 0.
$$
We also collect the  dependence of the constant $c(\cH,R_0)$ on $\cH$ and $R_0$. It will be shown that $c(\cH,R_0)$ only depends on the numbers $R_0, k, \alpha,$ $\mathbf{H}_\alpha(\D(0, 8R_0))$.

Let us make a counter-assumption: there exist $z_n\in\D(0, R_0)$ and normalised homeomorphic solutions $f_n$ to the nonlinear Beltrami equations of the type \eqref{hqrjac} with the regularity \eqref{Hcondition}, i.e.,
$$
\partial_{\zbar}f_{n}(z) = \cH_{n}(z, \partial_z f_{n}(z)) \qquad \text{a.e.},
$$
 such that
$$
J(z_n, f_n) \leq \frac{1}{n}.
$$
In particular, the H\"older constant  $\mathbf{H}_\alpha(\D(0, 8R_0))$, the H\"older exponent $\alpha$ and the ellipticity $k$ are assumed to be the same for each $\cH_n$.

Now, we may pass to the subsequence, if necessary, to assume that $z_{n} \to z_{\infty} \in \overline{\D(0, R_0)}$ and as a normalised family of quasiconformal maps $f_{n} \to f_{\infty}$  locally uniformly, where $f_\infty$ is quasiconformal and $f_\infty(0) = 0$, $f_\infty(1) = 1$ (see the Montel-type theorem \cite[Theorem 3.9.4]{AIM}). Moreover, by the Schauder norm estimate \eqref{thmnorm}, for any $R > 0$,
\begin{equation}\label{Dfgammab}
\aligned
\|D_z f_{n} \|_{C^{\gamma}(\D(0, R))} &\leq c\,\|D_z f_{n}\|_{L^{2}(\D(0, 2R))}\leq c\,\|f_{n}\|_{L^{2}(\D(0, 4R))}\\
&\leq c\,\eta_K(4R),
\endaligned
\end{equation}
where $c = c(\cH, R)$ and the second to the last inequality follows by Caccioppoli's inequality and the last one from the $\eta_K$-quasisymme\-try of quasiconformal maps.  Hence derivatives $D_z f_{n}$ have a local uniform $C^{\gamma}$-upper bound and mappings $f_{n}$ converge to $f_\infty$ in $C^{1, \gamma}_{\loc}(\C, \C)$, too.  Thus $J(z_\infty, f_\infty) = 0$.

We will show that the inconsistency follows from the fact that $f_\infty$ also solves a nonlinear Beltrami equation
\begin{equation}\label{Hinfty}
\partial_{\zbar}f_\infty(z) = \cH_{\infty}(z, \partial_z f_\infty(z)) \qquad \text{a.e.,}
\end{equation}
where $\cH_{\infty}$ will satisfy the assumption \eqref{Hcondition}.

We first find $\cH_\infty$ as a limit of the structure functions $\cH_{n}$. Namely, $\cH_{n}$ is locally uniformly equicontinuous on $\C \times \C$. Indeed, given open, bounded sets $\Omega'$, $\Omega''$ and $(z_i, \xi_i) \in \Omega' \times \Omega''$, by assumption
\begin{align*}
&|\cH_{n}(z_1, \xi_1) - \cH_{n}(z_2, \xi_2)| \leq \mathbf{H}_\alpha(\Omega')|z_1 - z_2|^\alpha(|\xi_1| + |\xi_2|) + k\,|\xi_1 - \xi_2|.
\end{align*}
This gives the equicontinuity. Thus passing to a subsequence it converges to a function $\cH_\infty$ locally uniformly, where $\cH_\infty$ has the same regularity and norm bounds \eqref{Hcondition} as the family $\cH_{n}$.


As $\cH_{\infty}$ has the required regularity properties, we must only show that $f_{\infty}$ satisfies equation \eqref{Hinfty}. But this is immediate from the fact that the convergence of $D_z f_{n}$ is also locally uniform (they converge in the H\"older class as seen above). By the earlier part of the proof of Theorem~\ref{Jac}, we now know that $J_{f_{\infty}} > 0$ in the set $\D(0,R_0)$, a contradiction to the fact that $J(z_\infty, f_\infty) = 0$. Hence there must be a lower bound for the Jacobian, and we have proved Theorem~\ref{Jac}.
\end{proof}

For the reader's interest, we also present a different proof for the positivity of the Jacobian in the autonomous case. This proof is based on Sto\"ilov factorisation, Hurwitz theorem and a compactness argument inspired by \cite{AstalaFaraco02}.

\begin{thm}\label{positivejacobianF}
Assume $\cH:\C \to \C$ is $k$-Lipschitz, where $k = \frac{K-1}{K+1} < 1$, with $\cH(0)=0$ and let $F \in W^{1, 2}_{\loc}(\Omega, \C)$ be a homeomorphic solution to  $$\partial_{\zbar} F(z) = \cH(\partial_z F(z)) \qquad {\text{a.e.}}$$
Then  $J(z,F) \neq 0$ at every point $z\in\Omega$.
\end{thm}

\begin{proof}
Let us fix a disk $\D(z_0,2R)\subset\Omega$ and a point $z_1\in \D(z_0,R)$ where $J(z_1,F)\neq 0$. The derivatives of $F$ are continuous by Proposition~\ref{ccc} and we can assume, for instance, that $\partial_xF(z_1)\neq0$ and we will show that $\partial_x F(z) \neq 0$ everywhere. This is enough, since $|D_z F|^2 \leq K J_F(z)$.

Let us define
\begin{equation} \label{yksi}
F_h(z)=\frac{F(z+h)-F(z)}{F(z_1+h)-F(z_1)},\qquad h>0.
\end{equation}
Clearly $F_h$ is well-defined on $\Omega_h = \{z\in\Omega: d(z,\partial\Omega)>h\}$, and $\D(z_0,2R)\subset\Omega_h$ for any $h<d(z_0,\partial\Omega)-2R$. Further, $F_h$ is $K$-quasiregular on $\Omega_h$, as we saw  in \eqref{qrforDF}. Moreover, by Proposition~\ref{ccc}, we know that $D_z F\in W^{1,2}_{\loc}(\Omega, \C)$ and $F \in C^{1,\frac1K}_{\loc}(\Omega)$.

We can factor, by Sto\"ilow factorisation,
$$F_h=H_h\circ\Phi_h$$
where $\Phi_h:\C\to\C$ is $K$-quasiconformal, and we choose the normalisation $\Phi_h(z_0)=0$, $\Phi_h(z_1)=1$, and $H_h:\Phi_h(\Omega_h)\to \C$ is holomorphic. Moreover, $H_h(1)=1$, by the definition of $F_h$ and the above normalisation of $\Phi_h$. Since $\Phi_h$ are normalised $K$-quasiconformal maps, there exists a limit $K$-quasiconformal map
$$\Phi=\lim_{h\to 0^+}\Phi_h,$$
with locally uniform convergence, at least for a subsequence, see the Montel-type theorem \cite[Theorem 3.9.4]{AIM}. Similarly, for the same subsequence $ \Phi_h^{-1} \to  \Phi^{-1}$ locally uniformly in $\Phi(\D(z_0,R))$.

Note further that since $F$ is continuously differentiable and $\partial_xF(z_1)\neq0$, the functions $F_h$ in \eqref{yksi} converge locally uniformly in $\D(z_0,R)$, hence also $H_h = F_h \circ \Phi_h^{-1}$ converges locally uniformly 
in $\Phi(\D(z_0,R))$.
Let us now fix a compact set $E\subset \Phi(\D(z_0,R))$ with $1$ as an interior point. Since $\Phi_h(\D(z_0,R))$ converges in the Hausdorff metric to $\Phi(\D(z_0,R))$, for every $h$ small enough  $\Phi(\D(z_0,R)) \Subset \Phi_h(\D(z_0,2R))$. Thus $E\subset\Phi(\D(z_0,R))\Subset\Phi_h(\Omega_h)$, and so $H_h$, $h < h_0$, is well-defined family of functions analytic  on a neighbourhood of $E$, with limit $$H=\lim_{h\to 0^+}H_h$$
at least for a subsequence. Of course, the limit mapping $H$ is holomorphic on a neighbourhood of $E$ and $H(1)=1$. Then it follows that
$$
\lim_{h\to 0^+}H_h\circ\Phi_h = H\circ\Phi
$$
uniformly on compact subsets of $\D(z_0,R)$. In particular, 
$$\frac{\partial_x F(z)}{\partial_x F(z_1)}=H(\Phi(z))\qquad\text{for every }z\in\D(z_0,R).$$ 

But the analytic functions $H_h$ do not have zeros in $\Phi_h(\Omega_h)$, since $F$ is a homeomorphism. By   the Hurwitz theorem $H$ as well is non-vanishing on $E$, that is, $\frac{\partial_x F(z)}{\partial_x F(z_1)}$ does not have zeros in $\D(z_0, R)$. We have shown our claim.
\end{proof}

\begin{rem}
Alternatively in the proof of Theorem~\ref{positivejacobianF} one can invoke the Hurwitz theorem for quasiregular mappings \cite{mini} which tells for any converging subsequence that either the limit
$\lim_j F_{h_j}(z)$ is non-vanishing everywhere, or the limit vanishes identically.
\end{rem}

\end{document}